\documentclass[reqno]{gokova}

\usepackage{amssymb,amsmath,latexsym,epsfig,amscd}
\usepackage[all]{xy}
\usepackage{psfrag}
\usepackage{amsfonts}
\usepackage{mathrsfs}
\usepackage{graphicx}
\usepackage{array}
\usepackage{multirow}
\usepackage{longtable}
\usepackage{color}
\usepackage{marvosym}
\usepackage{stmaryrd}
\usepackage{hyperref}

\newlength{\picwidth} 
\newlength{\miniwidth} 
\newlength{\nanowidth} 
\newlength{\melowidth} 
\newlength{\leftminiwidth} 
\newlength{\rightminiwidth} 
\newlength{\minipagewidth} 

\newcommand{\CC}{\mathbb{C}}

\newcommand{\PP}{\mathbb{P}}
\newcommand{\RR}{\mathbb{R}}
\newcommand{\cO}{\mathcal{O}}
\newcommand{\T}{\mycal{T}}

\newcommand{\ZZ}{{\mathbb Z}}
\newcommand{\CP}{\mathbb{C}\mathbb{P}}

\newcommand{\Fuk}{\text{\sf Fuk}}

\begin{document}
\def\E{\ifmmode{\mathbb E}\else{$\mathbb E$}\fi} 
\def\N{\ifmmode{\mathbb N}\else{$\mathbb N$}\fi} 
\def\R{\ifmmode{\mathbb R}\else{$\mathbb R$}\fi} 
\def\Q{\ifmmode{\mathbb Q}\else{$\mathbb Q$}\fi} 
\def\C{\ifmmode{\mathbb C}\else{$\mathbb C$}\fi} 
\def\H{\ifmmode{\mathbb H}\else{$\mathbb H$}\fi} 
\def\Z{\ifmmode{\mathbb Z}\else{$\mathbb Z$}\fi} 
\def\P{\ifmmode{\mathbb P}\else{$\mathbb P$}\fi} 
\def\T{\ifmmode{\mathbb T}\else{$\mathbb T$}\fi} 
\def\SS{\ifmmode{\mathbb S}\else{$\mathbb S$}\fi} 
\def\DD{\ifmmode{\mathbb D}\else{$\mathbb D$}\fi} 

\renewcommand{\a}{\alpha}
\renewcommand{\b}{\beta}
\renewcommand{\d}{\delta}
\newcommand{\D}{\Delta}
\newcommand{\e}{\varepsilon}
\newcommand{\g}{\gamma}
\newcommand{\G}{\Gamma}
\newcommand{\la}{\lambda}
\newcommand{\La}{\Lambda}
\newcommand{\n}{\nabla}
\newcommand{\var}{\varphi}
\newcommand{\s}{\sigma}
\newcommand{\Sig}{\Sigma}
\renewcommand{\t}{\tau}
\renewcommand{\th}{\theta}
\renewcommand{\O}{\Omega}
\renewcommand{\o}{\omega}
\newcommand{\z}{\zeta}

\newcommand{\ben}{\begin{enumerate}}
\newcommand{\een}{\end{enumerate}}
\newcommand{\be}{\begin{equation}}
\newcommand{\ee}{\end{equation}}
\newcommand{\bea}{\begin{eqnarray}}
\newcommand{\eea}{\end{eqnarray}}
\newcommand{\bc}{\begin{center}}
\newcommand{\ec}{\end{center}}

\newtheorem{thm}{Theorem}[section]
\newtheorem{cor}[thm]{Corollary}
\newtheorem{lem}[thm]{Lemma}
\newtheorem{prop}[thm]{Proposition}
\newtheorem{ax}{Axiom}
\newtheorem{conj}[thm]{Conjecture}

\theoremstyle{definition}
\newtheorem{defn}{Definition}[section]

\theoremstyle{remark}
\newtheorem{rem}{\rm\bfseries{Remark}}[section]
\newtheorem*{notation}{Notation}

\newtheorem{ques}{\rm\bfseries{Question}}[section]
\newtheorem{cons}[rem]{\rm\bfseries{Construction}}
\newtheorem{exm}[rem]{\rm\bfseries{Example}}


\newtheorem{question}[thm]{Question}

\address{
Ludmil Katzarkov\\
University of Vienna.
}
\email{lkatzark@math.uci.edu}

\address{
Victor Przyjalkowski\\
Steklov Mathematical Institute.
}
\email{victorprz@mi.ras.ru, victorprz@gmail.com}

\title[Landau--Ginzburg models --- old and new]{Landau--Ginzburg models --- old and new}

\author[KATZARKOV and PRZYJALKOWSKI]{ Ludmil Katzarkov, Victor Przyjalkowski}

\thanks{L.\,K was funded by NSF Grant DMS0600800, NSF FRG Grant DMS-0652633, FWF
Grant P20778, and an ERC Grant --- GEMIS,  V.\,P. was funded by FWF grant P20778,
RFFI grants 11-01-00336-a and 11-01-00185-a, grants MK$-1192.2012.1$,
NSh$-5139.2012.1$, and AG Laboratory GU-HSE, RF government
grant, ag. 11 11.G34.31.0023.
}

\begin{abstract}
In the last three  years  a new concept --- the concept  of  wall crossing
has emerged. The current situation with wall crossing phenomena, after
papers of Seiberg--Witten, Gaiotto--Moore--Neitzke, Vafa--Cecoti and
seminal works by Donaldson--Thomas, Joyce--Song,
Maulik--Nekrasov--Okounkov--Pandharipande, Douglas, Bridgeland, and
Kontsevich--Soibelman,  is very similar to the situation with Higgs
Bundles after the works of Higgs and Hitchin --- it is clear  that a
general ``Hodge type''  of theory exists and needs to be developed.
Nonabelian Hodge theory did  lead to strong mathematical applications~--- uniformization, Langlands program to mention a few. In the wall crossing
it is also clear that some ``Hodge type'' of theory   exists ---
Stability Hodge Structure (SHS). This theory  needs to be developed
in order to reap some
mathematical benefits --- solve long standing problems in algebraic
geometry.  In this paper we  look at SHS  from the perspective of
Landau--Ginzburg models and we  look  at some applications. We
consider simple  examples and explain some conjectures these examples
suggest.
\end{abstract}

\maketitle

\section{Introduction}

Mirror symmetry is a physical duality between $N= 2$ superconformal
field theories.  In the  1990's
Maxim Kontsevich reinterpreted this concept from physics as an
incredibly deep and far-reaching mathematical duality now known as
Homological Mirror Symmetry (HMS).  In a famous lecture in 1994, he
created a  frenzy in the mathematical community which lead to
 synergies between diverse mathematical disciplines:
symplectic geometry, algebraic geometry, and category theory.  HMS is
now the cornerstone of an immense field of active mathematical
research.

In the last three  years  a new concept --- the concept  of  wall crossing has emerged. The current situation with wall crossing phenomena, after  papers of Seiberg--Witten, Gaiotto--Moore--Neitzke, Vafa--Cecoti and seminal works by Donaldson--Thomas, Joyce--Song, Maulik--Nekrasov--Okounkov--Pandharipande, Douglas, Bridgeland, and Kontsevich--Soibelman,  is very similar to the situation with Higgs Bundles after the works of Higgs and Hitchin --- it is clear  that a general ``Hodge type''  of theory exists and needs to be developed. Nonabelian Hodge theory did  lead to strong mathematical applications --- uniformization, Langlands program to mention a few. In the wall crossing it is also clear that some ``Hodge type'' of theory needs to be developed in order to reap some
mathematical benefits --- solve long standing problems in algebraic geometry.

The foundations of these new Hodge structures, which we call \emph{Stability Hodge Structures (SHS)} will appear in a paper by the first author, Kontsevich, Pantev  and Soibelman ---
\cite{KKPS}.  In this paper we will look at SHS  from the perspective of Landau--Ginzburg models and we will also look  at some applications. We will consider simple  examples and explain some conjectures these examples suggest. The further elaboration and examples will appear in \cite{KKPS} and \cite{DKK}.

We start with the classical interpretation of wall crossings in Landau--Ginzburg models. After that we describe a hypothetical program of ``Stability Hodge Theory'' which combine Nonabelian and Noncommutative Hodge theory.  We consider some  possible applications  in this paper. First we consider an  approach to the conjecture  that the universal covering of a smooth projective variety is holomorphically convex.  This is  a classical question in algebraic geometry proven by the first author and collaborators for linear fundamental groups \cite{EKPR}. It was believed that for nonresidually finite fundamental groups one needs a different approach and in this paper we outline a procedure of extending the argument to the nonresidually finite case based on SHS. We also outline  possible applications to  Hodge structures with many filtrations and  to Sarkisov's theory.

Stability Hodge Structure is a notion which originates from functions of one complex variable  and combinatorics --- gaps,  polygons, and circuits. We give these classical  notions a new read through HMS and category theory, dressing them up with some cluster varieties and integrable systems.  After that we  enhance these data  additionally with some  basic nonabelian Hodge theory in order to  get  a property we need --- strictness.  In the same way as moduli spaces of Higgs bundles parameterize spectral coverings, the moduli space of deformed stability conditions parameterizes Landau--Ginzburg models.

We believe this is only the tip  of the iceberg and this very rich motivic conglomerate of ideas will play an important role in the  studies of categories and of algebraic cycles.
In particular we suggest that the   categorical notion of spectra can be seen as a Hodge theoretic notion related to the ``homotopy type of a category''.

\medskip

The paper is organized as follows. In Sections~\ref{section:clusters} and~\ref{section:Minkowski} we describe the classical approach to Landau--Ginzburg models and wall crossings. After that in Sections~\ref{section:wall crossings},~\ref{section:SHS},~\ref{section:Higgs bundles} we define Stability  Hodge Structures and build a parallel with Simpson's nonabelian Hodge  theory.
 We also discuss possible applications in Sections~\ref{section:spectra},~\ref{section:multi LG},~\ref{section:birational}.

\section{Classical  Landau--Ginzburg models and wall crossings }

\label{section:clusters}

In this section we recall  the ``classical'' way of interpreting wall crossing in the case of Landau--Ginzburg models.
We will establish a certain combinatorial framework on which we later base our constructions.

We recall the notion of  Landau--Ginzburg models from the Laurent polynomials point of view. For more details see, say,~\cite{Prz07} and references therein. Cluster transformations and Minkowski decompositions for Landau--Ginzburg models are discussed in~\cite{CCGGK12},~\cite{ACGK12},~\cite{CCGK13}.

Let  $X$ be a  smooth Fano variety  of dimension $n$. We can associate \emph{a quantum cohomology ring} $QH^*(X)=H^*(X,\Q)\otimes \Lambda$ to it, where $\Lambda$
is the Novikov ring for $X$. The multiplication in this ring, the so called \emph{quantum multiplication}, is given by \emph{(genus zero)
Gromov--Witten invariants} --- numbers counting rational curves lying in $X$. Given these data one can associate \emph{a regularized quantum differential operator} $Q_X$ (the second Dubrovin connection) --- the regularization of an operator associated with connection in the trivial vector bundle
given by a quantum multiplication by the canonical class $K_X$. In ``good'' cases such as we consider (for Fano threefolds or complete
intersections) the equation $Q_XI=0$ has a unique normalized analytic solution $I=1+a_1t+a_2t^2+\ldots$.

\begin{defn}
\label{definition: toric LG}
\emph{A toric Landau--Ginzburg model} is a Laurent polynomial $f\in \CC[x_1^{\pm 1}, \ldots, x_n^{\pm}]$ such that:
\begin{description}
  \item[Period condition] The constant term of $f^i\in \CC[x_1^{\pm 1}, \ldots, x_n^{\pm}]$ is $a_i$ for any $i$ (this means that $I$ is a period of a family $f\colon (\CC^*)^n\to \CC$, see~\cite{Prz07}).
  \item[Calabi--Yau condition] Any fiber of $f\colon (\CC^*)^n\to \CC$ after some fiberwise compactification has trivial dualizing sheaf.
  \item[Toric condition] There is an embedded degeneration $X\rightsquigarrow T$ to a toric variety $T$ whose fan polytope
  (the convex hull of generators of its rays) coincides with the Newton polytope (the convex hull of non-zero coefficients) of $f$.
A Laurent polynomial without the toric condition is called \emph{a weak Landau--Ginzburg model}.
\end{description}
\end{defn}

Toric Landau--Ginzburg models for complete intersections can be derived from the Hori--Vafa suggestions (see, say,~\cite{Prz09}).

\begin{defn}
Let $X$ be a general Fano complete intersection of hypersurfaces of degrees $d_1,\ldots,d_k$ in $\PP^N$.
Let $d_0=N-d_1-\ldots- d_k$ be its index. Then a Laurent polynomial
$$
f_X=\frac{(x_{1,1}+\ldots+x_{1,d_1-1}+1)^{d_1}\cdot\ldots\cdot(x_{k,1}+\ldots+x_{k,d_k-1}+1)^{d_k}}{\prod x_{ij}}+x_{01}+\ldots+x_{0d_0-1}.
$$
we call \emph{of Hori--Vafa type}.
\end{defn}


\begin{thm}[Proposition 9 in~\cite{Prz09} and Theorem 2.2,~\cite{IP11}]
The polynomial $f_X$ is a toric Landau--Ginzburg model for $X$.
\end{thm}

\begin{defn}
Let $f$ be a Laurent polynomial in $\C [x_0^{\pm 1},\ldots,x_n^{\pm 1}]$.
Then a (non-toric birational) symplectomorphism is called \emph{of cluster type}
if
 it is a composition of toric change of variables and symplectomorphisms of type
$$
y_0=x_0\cdot f_0(x_1,\ldots,x_i)^{\pm 1}, \ \ y_1=x_1,\ldots,y_n=x_n,
$$
for some Laurent polynomial $f_0$ and under this change of variables $f$ goes to a Laurent polynomial for which a Calabi--Yau condition holds.

It is called
\emph{elementary of cluster type} if (up to toric change of variables)
$$
f_0=x_1+\ldots+x_i+1.
$$
It is called \emph{of linear cluster type} if it is a
composition of elementary symplectomorphisms of cluster type and toric change of variables.
\end{defn}

\begin{rem}
For all examples in the rest of the paper the Calabi--Yau condition holds for all considered cluster type transformations.
\end{rem}

\begin{prop}
Let $f$ be a weak Landau--Ginzburg model for $X$. Let $f^\prime$ be a Laurent polynomial obtained from $f$ by symplectomorphism of
cluster type. Then $f^\prime$ is a weak Landau--Ginzburg model for $X$.
\end{prop}

\begin{proof}
A period giving constant terms of Laurent polynomials
is, up to proportion, an integral of the (depending on $\lambda\in \CC$) form $\frac{1}{1-\lambda f}\prod \frac{dx_i}{x_i}$
over a standard $n$-cycle on the torus $|x_1|=\ldots =|x_n|=1$. This integral does not change under cluster type symplectomorphisms.
\end{proof}

\begin{exm}
Let $X$ be a quadric threefold. There are two types of degenerations of $X$ to normal toric varieties inside the space
of quadratic forms. That is,
$$
T_0=\{x_1x_2=x_2^3\}\subset \PP[x_1:x_2:x_3:x_4:x_5]
$$
and
$$
T_1=\{x_1x_2=x_3x_4\}\subset \PP[x_1:x_2:x_3:x_4:x_5].
$$
Let
$$
f_0=\frac{(x+1)^2}{xyz}+y+z
$$
be a weak Landau--Ginzburg model of Hori--Vafa type for $X$. Let
$$
f_1=\frac{(x+1)}{xyz}+y(x+1)+z
$$
be its cluster-type transformation given by the change of variables
$$
\frac{y}{(x+1)}\mapsto y.
$$
One can see that $T_0=T_{f_0}$ and $T_1=T_{f_1}$.
\end{exm}

\begin{rem}
One can see that applying the same change of variables a second time to $f_1$ gives back (up to toric change of variables) $f_0$.
\end{rem}

\begin{exm}
Let $X$ be a cubic threefold. There are two types of degenerations of $X$ to normal toric varieties inside the space
of cubic forms. That is,
$$
T_0=\{x_1x_2x_3=x_4^3\}\subset \PP[x_1:x_2:x_3:x_4:x_5]
$$
and
$$
T_1=\{x_1x_2x_3=x_4^2x_5\}\subset \PP[x_1:x_2:x_3:x_4:x_5].
$$
Let
$$
f_0=\frac{(x+y+1)^3}{xyz}+z
$$
be a weak Landau--Ginzburg model of Hori--Vafa type for $X$. Let
$$
f_1=\frac{(x+y+1)^2}{xyz}+z(x+y+1)
$$
be its cluster-type transformation given by the change of variables
$$
\frac{z}{(x+y+1)}\mapsto z.
$$
One can see that $T_0=T_{f_0}$ and $T_1=T_{f_1}$.
\end{exm}

\begin{rem}
Applying this change of variables a second time to $f_1$  we  get (up to toric change of variables) $f_1$ again and applying it a third time we get  $f_0$ back.
\end{rem}

\begin{exm}
Let $X$ be a cubic fourfold. There are three types of degenerations of $X$ to normal toric varieties inside the space
of cubic forms. That is,
$$
T_{00}=\{x_1x_2x_3=x_4^3\}\subset \PP[x_1:x_2:x_3:x_4:x_5:x_6],
$$
$$
T_{10}=\{x_1x_2x_3=x_4^2x_5\}\subset \PP[x_1:x_2:x_3:x_4:x_5:x_6],
$$
and
$$
T_{11}=\{x_1x_2x_3=x_4x_5x_6\}\subset \PP[x_1:x_2:x_3:x_4:x_5:x_6],
$$
Let
$$
f_{00}=\frac{(x+y+1)^3}{xyzt}+z+t
$$
be a weak Landau--Ginzburg model of Hori--Vafa type for $X$. Let
$$
f_{10}=\frac{(x+y+1)^2}{xyzt}+z(x+y+1)+t
$$
be its cluster-type transformation given by the change of variables
$$
\frac{z}{(x+y+1)}\mapsto z
$$
and let
$$
f_{11}=\frac{(x+y+1)}{xyzt}+z(x+y+1)+t(x+y+1)
$$
be the cluster-type transformation of $f_{10}$ given by the change of variables
$$
\frac{t}{(x+y+1)}\mapsto t.
$$
One can see that $T_{00}=T_{f_{00}}$, $T_{10}=T_{f_{10}}$, and $T_{11}=T_{f_{11}}$.
\end{exm}

\begin{rem}
Applying the first change of variables a second time to $f_{10}$ we get (up to toric change of variables) $f_{10}$ again, applying it once more
we get  $f_{00}$, and applying any change of variables to $f_{11}$ we  get  $f_{10}$.
\end{rem}

\begin{exm}
Consider quadrics in $\PP=\PP(1,1,1,1,2)$. Denote the coordinates in $\PP$ by $x_0$, $x_1$, $x_2$, $x_3$, $x_4$, where the weight
of $x_4$ is 2.
The general quadric is
$$
T_1=\{F_2(x_0,x_1,x_2,x_3)+\lambda x_4=0\},
$$
where $F_2$ is a quadratic form and $\lambda\in \CC\setminus 0$.
Projection on the hyperplane generated by $x_0,\ldots,x_3$ gives an isomorphism of $T_1$ with $\PP^3$.
The general variety with $\lambda=0$ is a toric variety
$$
T_2=\{x_0x_1=x_2x_3\}.
$$
It degenerates to
$$
T_3=\{x_1x_2=x_0^2\}.
$$
One can see that $T_3$ is an image of $\PP(1,1,2,4)$ under the Veronese map $v_2$.

Consider the following 3 weak Landau--Ginzburg models for $\PP^3$:
$$
f_1=x+y+z+\frac{1}{xyz},
$$
$$
f_2=x+\frac{y}{x}+\frac{z}{x}+\frac{1}{xy}+\frac{1}{xz},
$$
$$
f_3=\frac{(x+1)^2}{xyz}+\frac{y}{z}+z.
$$
Changing toric variables one can rewrite $f_1$ as
$$
f_1^{\prime }=z(x+1)+y+\frac{1}{xyz^2},
$$
$$
f_1^{\prime \prime}=z(x+1)+\frac{y}{z}+\frac{1}{xyz}.
$$
The cluster-type change of variables
$$
x\mapsto x,\ \ y\mapsto y,\ \ z(x+1)\mapsto z
$$
sends $f_1^\prime$ to a Laurent polynomial that differs from $f_3$ by a toric change of variables and $f_1^{\prime \prime}$ to a polynomial
$$
z+\frac{(x+1)y}{z}+\frac{(x+1)}{xyz},
$$
which differs from $f_2$ by toric change of variables.
The cluster-type change of variables
$$
x\mapsto x,\ \ y(x+1)\mapsto y,\ \ z\mapsto z
$$
sends the last expression to $f_3$.

One can see that $T_1=T_{f_1}$, $T_2=T_{f_2}$, and $T_3=T_{f_3}$.
\end{exm}

\begin{thm}[Hacking--Prokhorov,~\cite{HaPr05}]
Let $X$ be a degeneration of $\PP^2$ to a $\Q$-Gorenstein surface with quotient singularities. Then $X=\PP(a^2,b^2,c^2)$,
where $(a, b, c)$ is any solution of the Markov equation $a^2+b^2+c^2=3abc$.
\end{thm}

\begin{rem}
All Markov triples are obtained from the basic one $(1,1,1)$ by a sequence of \emph{elementary transforms}
$$
(a,b,c)\mapsto (a,b, 3ab-c).
$$
\end{rem}

\begin{prop}[\cite{GU10}, see also~\cite{CMG13}]
\label{proposition:Galkin}
Let $(a,b,c)$ be a Markov triple and let $f$ be a weak Landau--Ginzburg model for $\PP^2$ such that $T_f=\PP(a^2,b^2,c^2)$.
Then there is an elementary cluster-type transformation such that for the image $f^\prime$ of $f$ under this transformation $T_{f^\prime}=\PP(a^2,b^2,(3ab-c)^2)$.
\end{prop}

{\bf Sketch of the proof}\ (S.\,Galkin).
Consider $d\geq c$ such that $3ad=b\ (\mathrm{mod}\ c)$. One can check that we can choose toric coordinates $x$, $y$ such that
in these coordinates vertices of the Newton polytope of $f$ are $(d,c)$, $(d-c,c)$, and $(-\frac{d(3ab-c)-b^2}{c},-3ab+c)$.
Let $p$ be the $k$-th integral point from the end of an edge of integral length $n$ of the Newton polytope of $f$. Then
the coefficient of $f$ at $p$ is $\binom{n}{k}$ (this can be proved by induction).
This means that
$$
f=x^{d-c}y^c(x+1)^c+\frac{1}{x^{\frac{d(3ab-c)-b^2}{c}}y^{3ab-c}}+\sum_r y^{-n_r}f_r(x),
$$
where $n_i$'s are non-negative and $f_i$'s are some Laurent polynomials in $x$.
One can check that the change of variables of cluster type
$$
y^\prime=y(x+1),\ \ x^\prime=x
$$
sends $f$ to a weak Landau--Ginzburg model $f^\prime$ such that $T_{f^\prime}=\PP(a^2,b^2,(3ab-c)^2)$. \qed

We extend  observed connection between degenerations and birational transformations further to  a general connection between geometry of moduli space of Landau--Ginzburg models, birational and symplectic geometry. We summarize this connection in Table~\ref{tab:Section 2 tab 1} and we will investigate it (mainly conjecturally) in the sections that follow.

\begin{table}[h]
\begin{center}
\begin{tabular}{|c||c|c|}
\hline
\begin{minipage}[c]{0.5in}
\centering
\medskip

\medskip

\end{minipage}
&
\begin{minipage}[c]{2in}
\centering
\medskip

Fano variety $X$

\medskip

\end{minipage}
&
\begin{minipage}[c]{2in}
\centering
\medskip

Landau--Ginzburg model $LG(X)$

\medskip

\end{minipage}
\\\hline\hline
\begin{minipage}[c]{0.5in}
\centering
\medskip

A side

\medskip

\end{minipage}
&
\begin{minipage}[c]{2in}
\centering
\medskip

$\Fuk (X)$: symplectomorphisms
and general degenerations

\medskip

\end{minipage}
&
\begin{minipage}[c]{2in}
\centering
\medskip

$FS(LG(X))$: degenerations

\medskip

\end{minipage}
\\\hline
\begin{minipage}[c]{0.5in}
\centering
\medskip

B side

\medskip

\end{minipage}
&
\begin{minipage}[c]{2in}
\centering
\medskip

$D^b_{sing}(LG(X))$: phase changes

\medskip

\end{minipage}
&
\begin{minipage}[c]{2in}
\centering
\medskip

$D^b(X)$: birational transformations

\medskip

\end{minipage}
\\\hline
\end{tabular}
\end{center}
\caption{Wall crossings.}
\label{tab:Section 2 tab 1}
\end{table}

\section{Minkowski decompositions and cluster transformations}

\label{section:Minkowski}

\begin{defn}
Let $N\cong \ZZ^n$ be a lattice. Denote $N_\RR=N\otimes \RR$. \emph{A polytope} $\Delta \subset N_\RR$ is a convex hull of finite number
of points in $N_\RR$. A polytope is called \emph{integral} iff these points lie in $N\otimes 1$. A polytope called \emph{primitive} if it is
integral and its vertices are primitive. A Laurent polynomial is called \emph{primitive} if its Newton polytope is primitive.
\end{defn}

\begin{defn}
\emph{The Minkowski sum} $\Delta_1+\ldots+\Delta_k$ of polytopes $\Delta_1,\ldots,\Delta_k$ is the polytope $\{v_1+\ldots+v_k| v_i\in \Delta_i\}$.
An integral polytope is called \emph{irreducible} if it can't be presented as a Minkowski sum of two non-trivial integral polytopes.
\end{defn}

\begin{rem}
A Minkowski sum of integral polytopes is integral.
\end{rem}

\begin{defn}
Consider an integral polytope $\Delta\in \ZZ^n$. \emph{A Minkowski presentation} of $\Delta$ is a presentation of each of its faces as
a Minkowski sum of irreducible integral polytopes such that if a face $\Delta'$ lies in a face $\Delta$ then the
intersections of Minkowski summands for $\Delta$ with $\Delta'$ give a presentation for $\Delta'$.

Consider a Laurent polynomial $f\in\CC[\ZZ^n]$. For any face $\Delta$ of $\Delta_f$ denote the sum
 of all monomials of $f$ lying in $\Delta$ by $f_\Delta$. The polynomial $f$ is called \emph{Minkowski polynomial} if
there exists a Minkowski presentation such that, for any face $\Delta$ of $\Delta_f$ with given Minkowski sum expansion $\Delta=\Delta_1+\ldots+\Delta_k$,
there are
Laurent polynomials $f_{\Delta_i}\in \CC[\ZZ^n]$ such that the coefficients of $f_{\Delta_i}$ at vertices of $\Delta_i$ are 1's
and $f_\Delta=f_{\Delta_1}\cdot \ldots\cdot f_{\Delta_k}$.
\end{defn}

\begin{rem}
Let $e$ be an edge of a Minkowski Laurent polynomial of integral length $n$.
Its unique Minkowski expansion to irreducible summands is the expansion to $n$
segments of integral length 1.
Thus the coefficient of the monomial associated to the $i$'th integral point of $e$
(from any end) is $\binom{n}{i}$.
\end{rem}


\begin{rem}
Toric Landau--Ginzburg models of Hori--Vafa type or toric Landau--Ginzburg models from~\cite{Prz09} are Minkowski Laurent polynomials.
\end{rem}

\begin{exm}[Ilten--Vollmert construction,~\cite{IV09}]
\label{example:Ilten-Vollmert}
Consider an integral polytope $\Delta\subset N=\ZZ^n$. Let the origin of
$N$ lie strictly inside $\Delta$.
Let $X=T_\Delta$ be the toric variety whose fan is the face fan for $\Delta$. Denote the dual lattice to $N$ by $M=N^\vee$.
Put $N'=N\oplus \ZZ$, $M'=M\oplus \ZZ$.
Let $C$ be the cone generated by $(\Delta,1)$. Then $X=Proj\,\CC[C^\vee\cap M']$ with
grading given by $d=(0,1)\in M'$.
For any primitive $r\in M'$, consider the map $r\colon N'\to\ZZ$. Let
$L_r=ker(r)$. Let $s_r$ be a retract (cosection) of the inclusion $i\colon L_r\to N'$, that is, a map $N'\to L_r$ such that
$s_ri=Id_{L_r}$. It is unique up to translations along $L_r$. Let $C^+=s_r(\{p\in C|\langle p, r\rangle=1\})$, $C^-=s_r(\{p\in C|\langle p, r\rangle=-1\})$ be two ``slices'' of $C$ cut out by evaluating function at $r$.

Choose $r$ such that $r=(r_0,0)\in M'$ and such that $C^-$ is a cone
with its single vertex a lattice point. Consider a Minkowski decomposition
$C^+=C_1+C_2$ to (possibly rational) polytopes
such that  for any vertex $v$ of $C^+$, at least one of
the corresponding vertices in $C_1$ and $C_2$ is a lattice point.
Let $D$ be the cone in $L_r\oplus \ZZ$ generated by $(C^-,0)$, $(C_1,1)$, and
$(C_2,-1)$. Denote $X'=Proj\,\CC[D^\vee\cap (L_r\oplus \ZZ)^\vee)]$
where the grading is now given by $(s_r(d),0)$.
\end{exm}

\begin{prop}[Remark 1.8 and Theorem 4.4 in~\cite{IV09}]
There is an embedded degeneration of $X'$ to $X$.
\end{prop}


\begin{exm}[Ilten]
\label{example:Ilten}
Let $\Delta\subset \ZZ^2$ be the convex hull of the points $(-1,2)$, $(1,2)$, and $(0,-1)$.
Then $X=\PP(1,1,4)$. Let $r=(0,1,0)$. Then $s_r$ is given by the matrix
$$
\left(
  \begin{array}{ccc}
    1 & 0 & 0 \\
    0 & 1 & 1 \\
  \end{array}
\right).
$$
We are in setup of Example~\ref{example:Ilten-Vollmert} (see Figure~\ref{figure:Ilten's example}). The vertex of $\{p\in C|\langle p, r\rangle=-1\}$
is $(0,-1,1)$ and goes to a vertex $(0,0)$ under $s_r$ and the vertices of $\{p\in C|\langle p, r\rangle=1\}$
are $(\pm \frac{1}{2},1,\frac{1}{2})$ and goes to vertices $(\pm \frac{1}{2},\frac{3}{2})$ under $s_r$.
That is, we have a Minkowski decomposition drawn on Figure~\ref{figure:decomposition}.

\begin{figure}[htb]
\includegraphics[width=2\nanowidth]{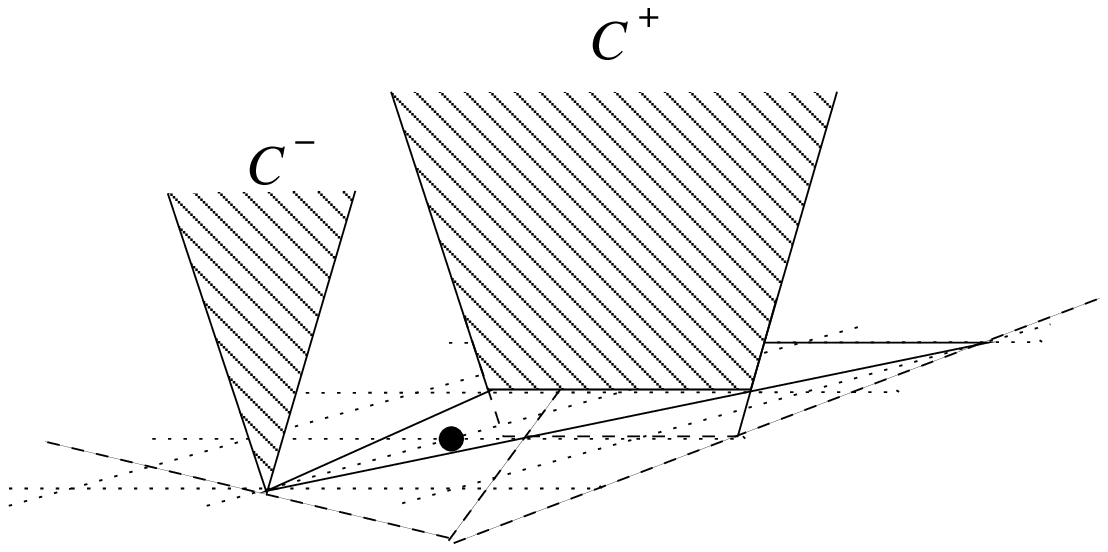}
\caption{Deformation of $\PP(1,1,4)$.}\label{figure:Ilten's example}
\end{figure}

\begin{figure}[htb]
\includegraphics[width=2.5\nanowidth]{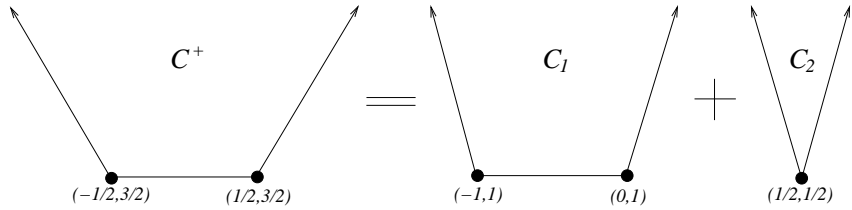}
\caption{Decomposition of $C^+$.}\label{figure:decomposition}
\end{figure}

The polytope for $X'$ is a convex hull of points $(-1,1)$, $(0,1)$, and $(1,-2)$ since the second coordinate becomes to be equal to 1 not on $(C_2,-1)$ but on $(2C_2,-2)$. Its face fan is a fan of $\PP^2$.
Thus we get a deformation of $\PP^2$ to $\PP(1,1,4)$.
\end{exm}

The following proposition shows that the degenerations given by Example~\ref{example:Ilten-Vollmert} give cluster transformations
for Minkowski polynomials.


\begin{prop}
\label{proposition:Minkowski-cluster}
Let $\Delta=\Delta_f$ be the Newton polytope of a Minkowski polynomial $f$. Let $\Delta'$ be a polytope obtained from $\Delta$
by the procedure described in Example~\ref{example:Ilten-Vollmert} given by integral Minkowski summands
agreing with the Minkowski decompositions of the faces of $\Delta$.
Then $\Delta'=\Delta_{f'}$ for some Minkowski polynomial $f'$.
\end{prop}

\begin{proof}
Let $f\in \CC[x_0^{\pm 1},\ldots,x_n^{\pm 1}]$. After toric changes of variables we can assume that $s_r$ is the projection on
coordinates $x_1,\ldots,x_n$. Then
$$f=f_+(x_1,\ldots,x_n)x_0+f_0(x_1,\ldots,x_n)+\frac{f_{-}(x_1,\ldots,x_n)}{x_0}.$$
As $f$ is a Minkowski polynomial we have $f_+=f_1f_2$. Thus after change of variables $x_0\to x_0/f_2$ we get a Minkowski
polynomial
$$f'=f_1(x_1,\ldots,x_n)x_0+f_0(x_1,\ldots,x_n)+\frac{f_{-}(x_1,\ldots,x_n)f_{2}(x_1,\ldots,x_n)}{x_0}$$
with Newton polytope $\Delta'$.
\end{proof}

\begin{rem}
Example~\ref{example:Ilten} shows that the statement of Proposition~\ref{proposition:Minkowski-cluster} holds for non-integral case as well.
This example is the first non-trivial cluster transformation given by Proposition~\ref{proposition:Galkin}.
\end{rem}

\begin{exm}
Let $\Delta$ be the convex hull of points $(-1,1)$, $(1,1)$, and $(0,-1)$. Then $X$ is a quadratic cone $\PP(1,1,2)$.
(A unique) Minkowski polynomial for $\Delta$ is
$$f=\frac{(x+1)^2y}{x}+\frac{1}{y}.$$
After cluster change of variables $y\to \frac{y}{x+1}$ we get a polynomial
$$
\frac{(x+1)y}{x}+\frac{x+1}{y}.
$$
It is (a unique) Minkowski polynomial for polytope $\Delta'$ --- the convex hull of points $(-1,-1)$, $(0,-1)$, $(1,-1)$, and $(0,-1)$.
These points generate the fan of a smooth quadric $X'$.
\end{exm}

\section{Degenerations and wall crossings}

\label{section:wall crossings}

In the previous section we have established certain combinatorial structures --- cluster transformations connected  to wall crossings.
We will relate  these  combinatorial structures   to moduli space of stability conditions. We do this in two steps:

{\bf Step 1.} First  we relate  the combinatorial structures  to  ``moduli space of Landau--Ginzburg models''.

{\bf Step 2.} Next  we describe hypothetically how the
``moduli space of Landau--Ginzburg models'' fits  in a ``twistor family'' with generic fiber the moduli space of stability conditions of a Fukaya--Seidel category.

 We  start with  step one --- collecting all Landau--Ginzburg models in a moduli space. The idea is  to record  wall crossings as relations in the mapping class group and then relations between relations and so on. This suggests a connection with Hodge theory and higher category theory. We will start with
a rather simple approach which we will enhance later in order to serve our purposes.
Nearly ten years ago it was discovered that, while the symplectic mapping class group of a curve equals the ordinary (oriented) mapping class group, these two groups differ greatly for higher dimensional symplectic manifolds. Understanding the structure of these groups has been a goal of many researchers in symplectic geometry. The initial  purpose of the construction below   was  to obtain a presentation of the symplectic mapping class group of toric hypersurfaces. Along the way we have obtained a characterization of the zero fiber of a
Stability Hodge Structure.

To explain our approach, we recall some notation and constructions. Assume $A \subset \Z^d$ is a finite set, $X_A$ the polarized toric variety associated to $A$ with ample line bundle $\mathcal{L}$. In \cite{GKZ}, the secondary polytope $\textrm{Sec} (A)$ parameterizing regular subdivisions was constructed and shown to be the Newton polytope of the $E_A$ determinant (a type of discriminant). We realize the toric variety associated to $\textrm{Sec} (A)$ as the coarse moduli space of a stack $\mathcal{X}_{\textrm{Sec}(A)}$ defined in \cite{Lafforgue}. We observe that the stack $\mathcal{X}_{\textrm{Laf}(A)}$  constructed in \cite{Lafforgue} has a proper map $\pi$ to $\mathcal{X}_{\textrm{Sec}(A)}$ whose fibers are degenerations of $X_A$, and we constructed a polytope $\textrm{Laf} (A)$ which is dual to the fan defining $\mathcal{X}_{\textrm{Laf} (A)}$. The zero set $\mathcal{H}_{\textrm{Sec}(A)}$ of a section of the associated line bundle parameterizes sections of $\mathcal{L}$ and degenerated sections are hypersurfaces in the associated degenerated toric variety. Upon restriction, we obtain a proper map $\pi : \mathcal{H}_{\textrm{Sec}(A)} \to \mathcal{X}_{\textrm{Sec}(A)}$ with non-singular fibers symplectomorphic to any non-degenerate section of $\mathcal{L}$.

Since $ \pi: \mathcal{H}_{\textrm{Sec}(A)} \to \mathcal{X}_{\textrm{Sec}(A)}$ is a proper map, we may consider symplectic parallel transport of the non-singular fibers along paths in the complement of the  zero set $Z_A$ of the $E_A$ determinant.
Denote by $H_p$ the fiber of $\pi$.
We observe that the subset of the fibers meeting the toric boundary $H$ are horizontal in the sense that if $q \in \partial H_p$ then the symplectic orthogonal $(T_q H_p)^{\perp_\omega} \subset T_q (\partial \mathcal{H}_{\textrm{Sec}(A)} )$, where $\omega$ corresponds to a restriction of Fubini--Study metric. That is, parallel transport is a symplectomorphism that preserves the boundary of the hypersurfaces.  Choosing a base point $p$ of $\mathcal{X}_{\textrm{Sec}(A)}\setminus Z_A$, we obtain a map from the based loop space $\rho :\Omega (\mathcal{X}_{\textrm{Sec}(A)} \setminus
 Z_A ) \to \textrm{Symp}^\partial ( H_p )$ and a group homomorphism
\begin{equation*} \rho_* : \pi_1 (\mathcal{X}_{\textrm{Sec}(A)}\setminus Z_A ) \to \pi_0 (\textrm{Symp}^\partial (H_p )), \end{equation*}
where $\pi_0 (\textrm{Symp}^\partial (H_p ))$ is a mapping class group.
However, from a field theory perspective, this homomorphism in imprecise; one should consider not only symplectomorphisms preserving the boundary, but also those that preserve the normal bundle of the boundary. In this way, we can glue two hypersurfaces together without creating an ambiguity in the symplectomorphism groups. We call such a symplectomorphism \emph{boundary framed morphism} and denote the corresponding group $\textrm{Symp}^{\partial,\textrm{fr}} (H_p)$. For toric hypersurfaces, this group is a central extension of $\textrm{Symp}^\partial (H_p)$. It is not generally the case, however, that parallel transport preserves the framing, but the change in framing can be controlled by keeping track of the homotopies in $\Omega (\mathcal{X}_{\textrm{Sec}(A)}\setminus Z_A)$ or by passing to the loop space of an auxiliary real torus bundle
$\mathcal{E} \to \mathcal{X}_{\textrm{Sec}(A)}\setminus Z_A$, giving a homomorphism
\begin{equation*} \tilde{\rho}_* : \pi_1 (\mathcal{E} ) \to \pi_0 (\textrm{Symp}^{\partial, \textrm{fr}} (H_p )). \end{equation*}
In many cases, this homomorphism is surjective.

The stack $\mathcal{X}_{\textrm{Sec}(A)}$ is as complicated combinatorially as the secondary polytope $\textrm{Sec} (A)$, which is computationally expensive to describe. While the Newton polytope of $E_A$ was found in \cite{GKZ}, $Z_A$ is far from smooth and there are open questions about its singular structure. We bypass these difficulties by considering only the lowest dimensional boundary strata of $\mathcal{X}_{\textrm{Sec}(A)}$ where non-trivial behavior occurs. Thus the first and main case we examine are the one dimensional boundary strata of $\mathcal{X}_{\textrm{Sec}(A)}$. Combinatorially, these are known as circuits.

\emph{A circuit $A$} is a collection of $d + 2$ points in $\Z^d$, such that there are exactly two coherent triangulations (see~\cite{GKZ}) of $A$, so the secondary polytope is a line segment and the secondary stack a weighted projective line $\mathbb{P} (a, b)$. $Z_A$ is either two or three points;  two of the points are the equivariant orbifold points $\{0, \infty\}$ and the possible third is an interior point. Both the constants, $a, b$ and the number of points in $Z_A$ depends on the convex hull and affine positioning of $A$ --- for more details see
\cite{DKK}.
When $Z_A$ consists of three points, their complement retracts  onto a  figure eight and the fundamental  group is free on two letters. In  this case, we have the based loops $\delta_1, \delta_2,  \delta_3 = \delta_2^{-1} \delta_1^{-1}$ encircling the three points. The symplectic monodromy  $T_i = \tilde{\rho_*} (\delta_i )$ is computable from known results in symplectic geometry as either spherical
Dehn twists  or as twists about a tropical decomposition. The image via $\tilde{\rho}_*$ gives the relation
\begin{equation} \label{eq:circ1} T_1 T_2 T_3 = T_{\partial H_p}, \end{equation}
where $T_{\partial H_p}$ is the central element determined by twisting the framing about the toric boundary. One of the most elementary examples is $X_A = \mathbb{P}^1 \times \mathbb{P}^1$ with polarization $\mathcal{O} (1,1)$, and the circuit is the four vertices of a unit square with the two diagonal triangulations. Here the hypersurface is $\mathbb{P}^1$ with four boundary points and the relation obtained above yields a classical relation in the mapping class group called the \emph{Lantern relation}.

When $Z_A$ consists of two points, one is an orbifold point and the other is a point with trivial stabilizer. If $\delta_1, \delta_2$ are
based paths encircling $Z_A$ and $T_1, T_2$ are the associated symplectomorphisms, we obtain a relation
\begin{equation} \label{eq:circ2} (T_1 T_2 )^a = T_{\partial H_p}. \end{equation}
A basic example of this relation arises as the homological mirror to $\mathbb{P}^2$ which is the set $A = \{(0, 0), (1, 0), (0, 1), (-1, -1)\}$. The constant $a$ occurring above is $3$ and the relation is in fact another classical mapping class group relation known as the \emph{star relation}.

We call the boundary framed, symplectic mapping class group relation occurring in equations \ref{eq:circ1} and \ref{eq:circ2} \emph{the circuit relation}. In general, any complex line in $\mathcal{X}_{\textrm{Sec}(A)}$ yields a relation in $\textrm{Symp}^{\partial, \textrm{fr}} (H_p)$ by homotoping the product of all the loops around the intersections with $Z_A$ to the identity. However, each such line can be degenerated to a chain of equivariant lines which are precisely circuits supported on $A$. Thus every relation obtained this way can be thought of as arising from a composition of circuit relations. As we saw in the previous two  sections Landau--Ginzburg mirrors of Fano manifolds  are fibrations of Calabi--Yau hypersurfaces. Therefore the above simple examples generalize to

\begin{thm}[\cite{DKK}]  Landau--Ginzburg mirrors of   Fano manifolds can be obtained  by a  superposition of circuits described above.
\end{thm}

Interpreting  Landau--Ginzburg  models as lines in the secondary stack we get

\begin{thm}[\cite{DKK}]  $\mathcal{X}_{\textrm{Sec}(A)}$ can be seen as moduli space of  Landau--Ginzburg models. In particular  some wall crossings correspond to passing through   $Z_A$.

\end{thm}

These two theorems complete Step 1.

\section{Wall crossings and Stability Hodge Structures}

\label{section:SHS}

We move to Step 2. building a ``twistor family'' with generic fiber the moduli space of stability conditions for  Fukaya--Seidel categories --- see \cite{AKO}.

{\bf Stability Hodge Structures. }
We start with Stability Hodge Structures, an artifact of  Donaldson--Thomas (DT) invariants. We will mainly consider Fukaya--Seidel categories but discussion in this section applies in  general.

The theory of Donaldson--Thomas invariants  and  wall crossing has become a central subject of Geometry and Physics. In a nutshell DT invariants are virtual numbers of stable objects in three dimensional Calabi--Yau category.
 Kontsevich and Soibelman  suggested Donaldson--Thomas invariants applicable to  triangulated category and Bridgeland stability conditions --- a refined version of so called \emph{motivic Donaldson--Thomas invariants} ---  MDT.
The wall crossing formulae (WCF) of MDT are expressed in terms of
factorization of quantum torus.
A  connection with nonabelian Hodge structures  comes naturally  here.
WCF for the Hitchin system  is connected  to ODE with small parameter and its asymptotic behavior. In fact the WCF relates  to Stokes data at infinity for this ODE and connects  with the   work of Ecalle and Voros on resurgence.

We will introduce a  new geometric structure which seems to be present in many of  above considerations  --- Stability Hodge Structures. These structures seem to have a huge potential of geometric applications some of which we discuss.

The moduli space of stability conditions of a category $C$   is very complicated with possibly fractional boundary. In the case of derived category of Calabi--Yau manifolds of dimension three and higher there is not any hypothetical description.
Still HMS predicts that the moduli space of mirror dual Calabi--Yau manifold
is embedded in locally closed cone in  moduli space of stability conditions of a category $C$. So it is a big open question how to characterize Hodge structures corresponding to mirror duals. Classically the moduli space of pure Hodge structures has a compactification by Mixed Hodge Structures (MHS). So it is natural to study limiting Donaldson--Thomas invariants and relate to WCF.

In the case of three-dimensional Calabi--Yau manifolds there are different types of MHS.
 The cusp case --- the deepest degeneration --- corresponds to a $t$-structures which is an extension of Tate motives. As a result we take a generating series of Donaldson--Thomas rank one torsion free invariants. It is expected that
in this case this generating series  (modulo  change of coordinates)  is the classical Gromov--Witten series which satisfies holomorphic anomaly equation. This translates into automorphic property for DT generating function. We expect that   automorphic property holds for higher ranks and plan to study it and show that WCF is necessary to assemble  limiting data.

A different MHS corresponds to conifold points and non maximal degeneration points. The wall crossings and DT data give a family of Integrable Systems in the following way. The vanishing   cycles $\Gamma_{short}$ and the monodromy define
a quotient category $\mathcal{T}/{\mathcal{A}}$ with the following sequence
on level of $K$--theory:
$$\Gamma_{short}\to K_0({\mathcal{T}})\to K_0({\mathcal{T}}/{\mathcal{A}}).$$

Using the Kontsevich--Soibelman noncommutative torus approach
we define a superscheme 
$$\mathbb G=\oplus_{p\in \Gamma_{short}}\mathbb G_p\to T_{non}.$$
Consider the zero grade $\mathbb G_0$ of
$\mathbb G$ over  $\Z$.
The global sections of 
$\mathbb G_0$
define Betti moduli space
 --- an integrable system 
$$\Gamma (G_0)=\oplus \mathcal O(M_j).$$

In order to consider the interaction with the rest of the category we include global WCF. In this case we obtain a torus action, which  produces a stack over  Betti moduli space:
$$X/(\CC^{*})^{\times n}\to M_1\times M_2\times \ldots \times M_k.$$

All these stacks fit in a  constructible sheaf.

To summarize we give a provisional definition, which covers the cases
of Bridgeland, geometric (volume forms), and generalized (log forms) stability conditions:

\begin{defn}

  \emph{Stability Hodge Structure (SHS)} for Fukaya--Seidel category $\mathcal F$ is the following data:

\begin{itemize}
  \item[i)] The moduli space of stability conditions $S$ for $\mathcal F$.
  \item[ii)] Divisor  $D$ at infinity giving a partial compactification of $S$ and  parametrizing the degenerated limiting stability conditions --- stability conditions for quotient categories, the category factored by the objects (vanishing cycles) on which stability conditions vanish.
  \item[iii)] Besides the degeneration we record the WCF --- all recorded together.
Over each point of $D$ we put Betti moduli space locally produced by WCF.
All these moduli space fit in a constructible sheaf over $S$.
\end{itemize}

\end{defn}

Let us illustrate these structures through two  examples.
We start with the category $\widetilde{\mathbb A}_2$ --- the Fukaya category of the conic bundle
$\{uv=y^2-x^3-a x-b\}$, $a,b\in \CC$. In this case, the Stability Hodge Structure is a
sheaf over $\CC^2$ with coordinates $a,b$.

\begin{figure}[htb]
\includegraphics[width=\nanowidth]{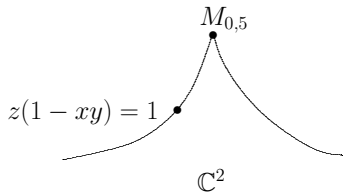}
\caption{Compactification of moduli space of stability conditions for a category $\widetilde{\mathbb A}_2$.}\label{figure:compactification A2}
\end{figure}

The points of the discriminant parameterize limiting stability
conditions. The fibers are Betti moduli spaces of vanishing cycles
which generically over the discriminant are the affine surface
$z(1-xy)=1$.  The special fiber over the cusp is the moduli space
$M_{0,5}$ of rank two bundles over projective line with one irregular
singularity and five Stokes directions at infinity (see Figure~\ref{figure:compactification A2}).

A different example is the  Fukaya--Seidel category  $\mathbb A_4$.
 We start with  a generic polynomial $p\in \CC[z]$ of
degree 5. It defines a Riemann surface $C=\{p(z)=w\}$ and 5:1-covering $\varphi \colon C\to \CC$. The ramification locus for $\varphi$ are 4 points $p_1,\ldots,p_4$---
roots of $p'$. Consider 4 paths $l_1,\ldots, l_4$ from $p_i$'s to infinity. The polynomial $p$ is generic, so the ramification is as simple as it can be and $\varphi^{-1}(l_i)$ are thimbles covering $l_i$'s 2:1. They generate a Fukaya category for $C$ and correspond to vertices of $\mathbb A_4$
quiver. ``Neighbor'' thimbles intersect at infinity: $i$-th one intersects $(i+1)$-th at one point. These intersections correspond to arrows between vertices in the quiver.

In this example the divisor $D$  at infinity parameterizes the semiorthogonal decompositions of the $\mathbb A_4$ category. The fibers of the constructible sheaf are moduli spaces of stability conditions for  $\mathbb A_3 \times  \mathbb A_1 $ categories.  Similarly on the singular points
of $D $ we get as fibers moduli spaces of stability conditions for  $\mathbb A_2 \times \mathbb A_2$ categories. This leads to a rich mixed Hodge theory structure associated with $D$ and monodromy action around it. In the next section we will see that in limit stability conditions behave as coverings so the above picture fits. This monodromy relates to the wall-crossings changes. In particular it sends the preferred set of thimbles generating the  $\mathbb A_4$ category from a generator consisting of the sum of $4$   thimbles
$G = L_1+L_2+L_3+L_4$  (with $Hom(L_i,L_{i+1})$  of rank 1) to $G' = L' + L_1 + L_3+ L_4$ by a mutation.  This mutation reduces
the generation time (see Section~\ref{section:spectra}) from $t(G)=3$ to $t(G')=2$.
We will represent is as an invariant of
of  Stability Hodge Structures in  Section~\ref{section:spectra}.

In the next section we build a twistor type of family where the generic fiber is a SHS.

\section{Higgs bundles and stability conditions --- analogy}

\label{section:Higgs bundles}

In this section we proceed describing the analogy  between
Nonabelian and Stability Hodge Structures.
We build the ``twistor'' family
so that the fiber over zero is the ``moduli space''  of Landau--Ginzburg models and the generic fiber is the Stability Hodge Structure defined above.

Noncommutative Hodge theory endows the cohomology groups of a dg-category with additional linear data --- the noncommutative Hodge structure --- which
records important information about the geometry of the
category. However, due to their linear nature, noncommutative Hodge
structures are not sophisticated enough to codify the full
geometric information hidden in a dg-category. In view of the
homological complexity of such categories it is clear that only a
subtler non-linear Hodge theoretic entity can adequately capture the
salient features of such categorical or noncommutative geometries. In this section by analogy with ``classical nonabelian Hodge theory''  we  construct and  study from such an prospective a new type of entity of exactly such
type --- the Stability Hodge Structure associated with a dg category.

As the name suggests, the SHS of a category is related to the
Bridgeland stabilities on this category.  The moduli space ${\sf
  Stab}_{C}$ of stability conditions of a triangulated dg-category
$C$ is, in general, a complicated curved space, possibly with
fractal boundary. In the special case when $C$ is the Fukaya
category of a Calabi--Yau threefold, the space ${\sf Stab}_{C}$ admits
a natural one-parameter specialization to a much simpler space ${\sf
  S}_{0}$. Indeed, HMS predicts that the moduli space of complex
structures on the  Calabi--Yau threefold maps to a Lagrangian
subvariety ${\sf Stab}^{\text{geom}}_{C} \subset {\sf Stab}_{C}$.
(Recall the holomorphic volume form and integrating it defines a stability condition and its charges,) The
idea is now to linearize ${\sf Stab}_{C}$ along ${\sf
  Stab}^{\text{geom}}_{C}$, i.e. to replace ${\sf Stab}_{C}$ with a
certain discrete quotient ${\sf S}_{0}$ of the total space of the
normal bundle of ${\sf Stab}^{\text{geom}}_{C}$ in ${\sf
  Stab}_{C}$. Specifically, by scaling the differentials and higher products  in $C$, one
obtains a one parameter family of categories $C_{\lambda}$ with
$\lambda \in \mathbb{C}^{*}$, and an
associated family ${\sf S}_{\lambda} := {\sf Stab}_{C_{\lambda}}$,
$\lambda \in \mathbb{C}^{*}$ of moduli of stabilities. Using
holomorphic sections with prescribed asymptotic at zero one can
complete the family $\{{\sf S}_{\lambda}\}_{\lambda \in
  \mathbb{C}^{*}}$ to a family ${\sf S} \to \mathbb{C}$ which in
a neighborhood of ${\sf Stab}^{\text{geom}}_{C}$ behaves like a
standard deformation to the normal cone. The space ${\sf S}_{0}$ is
the fiber at $0$ of this completed family and conjecturally ${\sf S}
\to \mathbb{C}$ is one chart of a twistor-like family $\mathcal{S} \to
\mathbb{P}^{1}$ which is by definition \emph{ the Stability Hodge
  Structure associated with $C$}.

Stability Hodge Structures are expected to exist for more general dg-categories, in particular for Fukaya--Seidel categories associated with
a superpotential on a Calabi--Yau space or with categories of
representations of quivers.  Moreover, for special non-compact
Calabi--Yau 3-folds, the zero fiber ${\sf S}_{0}$ of a Stability Hodge
Structure can be identified with the Dolbeault realization of a
nonabelian Hodge structure of an algebraic curve. This is an
unexpected and direct connection with Simpson's nonabelian Hodge theory which we exploit further suggesting some geometric applications.

 We briefly recall nonabelian Hodge theory settings. According to Simpson
we have one parametric twistor family such that the fiber over zero is the
moduli space of Higgs bundles and the generic fiber is the moduli space of
representation of the fundamental group --- $M_{Betti}$. 

In this section we state that we expect similar behavior of moduli space of Stability conditions. In other moduli space of stability conditions of Fukaya--Seidel category can be included in   one parametric twistor family, and we describe  fiber over zero in details in the next subsection.


We give an example:

\begin{exm}[twistor family for Stability Hodge Structures for the category $\mathbb A_n$]

We will give a brief explanation  the calculation of the
twistor family for the SHS for the category  $\mathbb A_n$. 
We start with the moduli space of stability conditions for
the category  $\mathbb A_n$, which can be identified with
differentials
$ e^{p}dz, $
where $p\in \CC[z]$ is a generic polynomial of degree $n+1$, see~\cite{KKPS}.

Let us denote one  holomorphic form  $ e^{p}dz $ by $Vol$. Locally there exist a holomorphic coordinate $w$  such that $Vol=dw.$  Geodesics in the metric $|Vol|^2$  are the straight real lines in coordinate $w$, the same as real lines on which $Vol$ has constant phase. Therefore  they are special Lagrangians for $Vol$  (and in fact for any real symplectic structure).

Observe that this geodesics  are asymptotic to infinity because the integral of
$|Vol|=e^{Re(p)}|dz|$  absolutely converges on them
 hence $Re(p)$ approaches infinity, as these lines are
noncompact in the uncompactified plane  $z$, and therefore  $|z|$ goes to
 infinity. To compensate infinite length in the usual metric $|dz|^2$  we use the fact that  $e^{Re(p)}$ converges  to zero  iff $Re(p)$  converges  to minus  infinity.

\end{exm}

So  after completion in the metric defined above (so the vertices are in the finite part now)  we enhance the polygon by assigning  angles and lengths. These enhanced polygons record our stability conditions. Indeed we have $(2(n+1)-3)$-dimensional space of polygons plus one global angle --- it is a real $2n$ dimensional space.  In Example~\ref{example:A2} we give a simple example the polygons for the category  $\mathbb A_2$ and a wall crossing phenomenon. The stable objects correspond to edges and diagonals. In the picture in the example we lose one stable object while crossing a wall.

\begin{exm}[stability for $\mathbb A_2$]
\label{example:A2}
For $\mathbb A_2$ category we have $\deg p=3$.
The left part of Figure~\ref{figure:A2} represents two of the stable objects for the $\mathbb A_2$ category. The third stable object is the third
edge of the triangle. The wall crossing makes the angle between the first two edges bigger then $\pi$ and as a result the third edge is not a stable object any more.

\begin{figure}[htb]
$$
\xymatrix{\includegraphics[width=0.6\nanowidth]{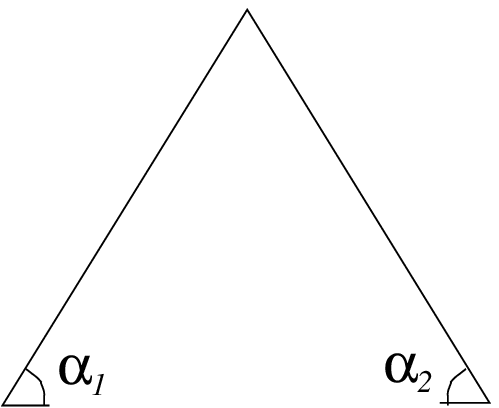}\ar[r] & \includegraphics[width=0.25\nanowidth]{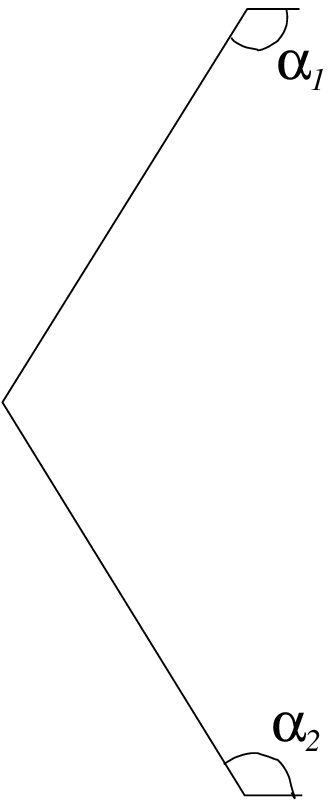}}
$$
\caption{Stability conditions for the $\mathbb A_2$ category.}\label{figure:A2}
\end{figure}

\end{exm}

Now we consider the ``twistor family'' --- the limit
$ e^{p(z)/u}dz, $
where $u$ is a complex number tending to 0.
Geometrically limit differential can be identified with graphs --- see Example~\ref{example:limit}.

\begin{exm}[limit $e^{p/u}dz$]
\label{example:limit}
Take a limit of $e^{p/u}dz$ with $u$ tending to zero.
The limits of polygons are graphs. We record the length, angle, and monodromy and this  defines a covering of the complex plane.
Thus this construction identifies a limit of moduli space of stability conditions for $\mathbb A_n$ category with some Hurwitz subspace ---
a subscheme of coverings.
In particular, these two spaces have the same number of components.
Figure~\ref{figure:limit} represents a procedure of associating the monodromy of the covering to the vertices of the graph.

\begin{figure}[htb]
\includegraphics[width=1.5\nanowidth]{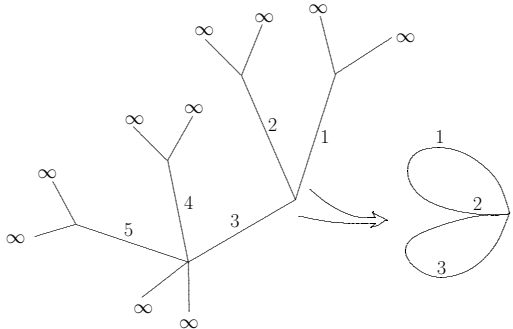}
\caption{Building coverings out of limit.}\label{figure:limit}
\end{figure}

\end{exm}



\begin{rem} Similarly one can compute the twistor family for the equivariant  $\mathbb A_n$ category and see appearance of gaps in spectra in connection with
the weight filtration of completions of special  local rings --- see Section~\ref{section:spectra}. Observe that idea of coverings brings the Fukaya category of a Riemann surface of genus $g$ very close to the    $\mathbb A_{2g+1}$ category. Also product of Fukaya categories of curves in combined  with Luttinger surgeries gives many opportunity for stability conditions with many components as well as many possibilities for the behavior of gaps and spectra.
The interplay between coverings and stability conditions suggests that one can have symplectic manifolds with the same Fukaya categories but different moduli of stability conditions. We conjecture that  the moduli spaces of coverings  obtained near different cusps being different algebraically should imply that this different manifolds are nonsymplectomorphic.

\end{rem}

{\bf The fiber over zero. }
\label{subsection:fiber over 0}
The fiber over zero (described in what follows)   plays an analogous role to the moduli space of Higgs bundles in Simpson's  twistor family in the theory of nonabelian Hodge structures.  Constructing it  amounts to a repetition of our construction in Section~\ref{section:wall crossings} from a new perspective and enhanced with more structure.

The     $\mathbb A_n$ example considered above is a simple example of more general Fukaya -Seidel categories that arise in Homological Mirror Symmetry.
Stability conditions associated to the Fukaya - Seidel category
are closely related to the complex deformation parameters, i.e. the
moduli space of Landau Ginzburg models. We begin by recalling the
general setup in the case of  Landau-Ginzburg models.
 The prescription  given by Batyrev, Borisov, Hori, Vafa in \cite{BB}, \cite{hori-vafa}
to obtain homological mirrors for toric Fano varieties is
perfectly explicit and provides a reasonably large set of examples
to examine. We recall that if $\Sigma$ is a fan in $\R^n$ for a
toric Fano variety $X_\Sigma$, then the homological mirror to the
B model of $X_\Sigma$ is a Landau--Ginzburg model $w : (\C^*)^n \to \C$ where
the Newton polytope $Q$ of $w$ is the convex hull of generators of rays
of $\Sigma$. In fact, we may consider the domain $(\C^*)^n$ to
occur as the dense orbit of a toric variety $X_A$, where $A$ is $Q
\cap \Z^n$ and $X_A$ indicates the polytope toric construction. In
this setting, the function $w$ occurs as a pencil $V_w \subset H^0
(X_A, L_A)$ with fiber at infinity equal to the toric boundary of
$X_A$. Similar construction works for generic non-toric Fanos.
In this paper we work with  directed Fukaya category associated to the
superpotential $w$ --- Fukaya--Seidel categories.   To build on the
discussion above, we discuss here Fukaya--Seidel  categories in the
context of stability conditions. The fiber over zero corresponds to the moduli of complex
structures. If $X_A$ is toric, the space of complex structures on
it is trivial, so the complex moduli appearing here are a result
of the choice of fiber $H \subset X_A$ and the choice of pencil
$w$ respectively. The appropriate stack parameterizing the choice
of fiber contains the quotient $[U/(\CC^*)^n]$ as an open dense
subset where $U$ is the open subset of $H^0 (X_A, L_A)$ consisting
of those sections whose hypersurfaces are nondegenerate (i.e.
smooth and transversely intersecting the toric boundary) and
$(\C^*)^n$ is acts by its action on $X_A$. To produce a reasonably
well-behaved compactification of this stack, we borrow from the
works of Alexeev (\cite{Al02}), Gelfand, Kapranov, and Zelevinsky (\cite{GKZ}), and
Lafforgue (\cite{Lafforgue}) to
construct the stack $\mathcal{X}_{\textrm{Sec} (A)}$ with universal
hypersurface stack $\mathcal{X}_{Laf (A)}$. We quote the following
theorem which describes much of the qualitative behavior of these
stacks:

\begin{thm}[\cite{DKK}]

\begin{itemize}
  \item[i)] The stack $\mathcal{X}_{\textrm{Sec} (A)}$ is a toric stack with moment polytope equal to the secondary polytope $Sec (A)$ of $A$.
  \item[ii)] The stack $\mathcal{X}_{Laf (A)}$ is a toric stack with moment polytope equal to the Minkowski sum $Sec (A) + \Delta_A$ where $\Delta_A$ is the standard simplex in $\R^A$.
  \item[iii)] Given any toric degeneration $F: Y \to \C$ of the pair $(X_A,
H)$, there exists a unique map $f : \C \to \mathcal{X}_{\textrm{Sec} (A)}$
such that $F$ is the pullback of $\mathcal{X}_{Laf (A)}$.
\end{itemize}

\end{thm}

We note that in the theorem above, the stacks $\mathcal{X}_{Laf
(A)}$ and $\mathcal{X}_{\textrm{Sec} (A)}$ carry additional equivariant
line bundles that have not been examined extensively in existing
literature, but are of great geometric significance. The stack
$\mathcal{X}_{\textrm{Sec} (A)}$ is a moduli stack for toric degenerations
of toric hypersurfaces $H \subset X_A$. There is a hypersurface
$\mathcal{E}_A \subset \mathcal{X}_{\textrm{Sec} (A)}$ which parameterizes
all degenerate hypersurfaces. For the Fukaya category of
hypersurfaces in $X_A$, the complement $\mathcal{X}_{\textrm{Sec} (A)}\setminus \mathcal{E}_A$ plays the role of the classical stability
conditions, while including $\mathcal{E}_A$ incorporates the
compactified version where MHS come into effect. We predict that
the walls of the stability conditions occurring in this setup are
seen as components of the tropical amoeba defined by the principal
$A$-determinant $E_A$.

To find the stability conditions associated to the directed Fukaya
category of $(X_A, w)$, one needs to identify the complex
deformation parameters  associated to this model. In fact, these are precisely
described as the coefficients of the superpotential, or in our
setup, the pencil $V_w \subset H^0 (X_A , w)$. Noticing that the
toric boundary is also a toric degeneration of the hypersurface,
we have that the pencil $V_w$ is nothing other than a map from
$\mathbb{P}^1$ to $\mathcal{X}_{\textrm{Sec} (A)}$ with prescribed point at
infinity. If we decorate $\mathbb{P}^1$ with markings at the
critical values of $w$ and $\infty$, then we can observe such a
map as an element of $\mathcal{M}_{0, Vol (Q) + 1}
(\mathcal{X}_{\textrm{Sec} (A)} , [w])$ which evaluates to $\mathcal{E}_A$
at all points except one and $\partial X_A$ at the remaining
point. We define the cycle of all stable maps with such an
evaluation to be $\mathcal{W}_A$ and regard it as the appropriate
compactification of complex structures on Landau--Ginzburg A-models. Applying
techniques from fiber polytopes we obtain the following
description of $\mathcal{W}_A$:

\begin{thm}[\cite{DKK}] The stack $\mathcal{W}_A$ is a toric stack with moment polytope equal to the monotone path polytope of $Sec(A)$. \end{thm}

The polytope occurring here is not as widely known as the
secondary polytope, but occurs in a broad framework of so called
iterated fiber polytopes introduced by Billera and Sturmfels.

In addition to the applications of these moduli spaces to
stability conditions, we also obtain important information on the
directed Fukaya categories and their mirrors from this approach.
In particular, the above theorem may be applied to computationally
find a finite set of special Landau--Ginzburg models $\{w_1, \ldots, w_s\}$
corresponding to the fixed points of $\mathcal{W}_A$ (or the
vertices of the monotone path polytope of $Sec (A)$). Each such
point is a stable map to $\mathcal{X}_{\textrm{Sec} (A)}$ whose image in
moment space lies on the $1$-skeleton of the secondary polytope.
This gives a natural semiorthogonal decomposition of the directed
Fukaya category into pieces corresponding to the components in the
stable curve which is the domain of $w_i$. After ordering these
components, we see that the image of any one of them is a
multi-cover of the equivariant cycle corresponding to an edge of
$Sec (A)$. These edges are known as circuits in combinatorics
(see \cite{DKK}).

Now we put this moduli space as a ``zero fiber'' of the twistor family of
moduli family of stability conditions.

We do this in two steps:

1. The following theorem suggests existence of a formal moduli space  $M$ of Landau--Ginzburg models $f\colon \overline{Y} \to \CP^1$.

\begin{thm}[see \cite{KKPS}] There exists a formal moduli space $M$   determined by the solutions of
the Maurer--Cartan equations for  dg-complex

$$
\begin{CD}
\cdots@<<< \Lambda^3 T_{\overline{Y}} @<<< \Lambda^2 T_{\overline{Y}} @<<< T_{\overline{Y}} @<<< \cO_{\overline{Y}} @<<< 0 \\[-3mm]
@. -3     @.             -2      @. -1       @. 0
\end{CD}
$$

\end{thm}

In the above complex the differential is $df$ and we can restate it by saying that  this complex determines  deformations of the Landau--Ginzburg  model, and these deformations are unobstructed. We also have a $\CC^*$-action on $M$ with fixed points corresponding to limiting stability conditions  --- see \cite{KKPS}.

Over the moduli space  $M$ defined above we have a variation of Hodge structures defined by the cohomologies of the perverse sheaf of vanishing cycles over $Y$. This defines local system $V$  over $M$ and its compactification.

\begin{conj}[see \cite{KKPS}] The  relative  completion with respect of $V$  in the fixed points of the  $\CC^*$-action on the compactification of $M$
 has a mixed Hodge structure.
\end{conj}

2. The above moduli space is too big.  So we will cut its dimension down to the moduli space of stability conditions. We introduce a new moduli space which embeds in $M$.

We   study deformations of $\overline{Y} \to \CP^1$  with ``fixing fiber at
infinity''. Deformation of a smooth variety $\overline{Y}$  with fixed $\CP^1$  is controlled by the following sheaf of dg Lie algebras on $\overline{Y}$:
$$T_{\overline{Y}} \to f^* T_{\CP^1}$$
(the differential is the tangent map).

By fixing the fiber at infinity we get  a subsheaf of dg Lie algebras
$$T_{\overline{Y},Y_\infty}  \to f^*T_{\CP^1, \infty}.$$

\begin{thm}[\cite{KKPS}]  A subsheaf of dg Lie algebras
$$T_{\overline{Y},Y_\infty}  \to f^*T_{\CP^1, \infty}.$$ determines 
smooth moduli stack.
Its dimension is equal to the dimension of the moduli space of  stability conditions.
\end{thm}

A geometric realization of this moduli space, which embeds in $M$  was described above. We will denote it  by   $M(\PP^1, CY)$ (or $M(\PP^k, CY)$ for
multipotential Landau--Ginzburg models).

\begin{rem} We can consider  bigger moduli space by fixing the vector fields only over a part of the  divisor at infinity. This corresponds to  taking a Landau--Ginzburg model through a point of non maximal degeneration. This defines a bigger moduli space of stability conditions with more stable objects.
\end{rem}

\begin{rem} The  moduli spaces we discuss could have many components. Such a phenomenon would have many interesting  implications. It produces possibilities of many new birational and symplectic invariants.
\end{rem}

In the same way as the fixed point set under the $\CC^*$-action plays an important role in describing the rational homotopy types of smooth projective varieties we study the fixed points of the  $\CC^*$-action on $F$ and derive information about the homotopy type of a category. In the rest of the paper we will denote
$\mathcal{W}_A$ by $M(\PP^1, \mathcal{X}_{\textrm{Sec} (A)} )$ (or $M(\PP^k, CY)$) in order to stress the connection with Landau--Ginzburg models
(here $CY$ denotes the moduli space of Calabi--Yau mirrors to the anticanonical section of the Fano manifold we consider).

After the journal version of this paper was published, the following related works were pointed out to us by P.\,Boalch:~\cite{BB04},~\cite{B07},~\cite{B12}.

\section{Spectra and holomorphic convexity}

\label{section:spectra}

In this section we explain briefly  how Orlov spectra
are related to Stability Hodge Structures.

Recall that noncommutative Hodge structures were introduced by  Kontsevich and Katzarkov and Pantev \cite{KKP1}  as  means of bringing
the techniques and tools of Hodge theory into the categorical and
noncommutative realm.  In the classical setting, much of the
information about an isolated singularity is recorded by means of
the Hodge spectrum, a set of rational eigenvalues of the monodromy
operator.  The Orlov spectrum (defined below), is a categorical
analogue of this Hodge spectrum appearing in the work of Orlov and
Rouquier. The missing numbers in
the spectra are called gaps.

Let $\mathcal T$ be a triangulated category.  For any $G \in
\mathcal T$ denote by $\langle G \rangle_0$ the smallest full
subcategory containing $G$ which is closed under isomorphisms,
shifting, and taking finite direct sums and summands. Now
inductively define $\langle G \rangle_n$ as the full subcategory
of objects, $B$, such that there is a distinguished triangle, $X
\to B \to Y \to X[1]$, with $X \in \langle G \rangle_{n-1}$ and $Y
\in \langle G \rangle_0$, and direct summands of such objects.

\begin{defn}
Let $G$ be an object of a triangulated category $\mathcal{T}$.  If
there is an $n$ with $\langle G \rangle_{n} = \mathcal T$, we set,
\begin{displaymath}
 t(G):= \text{min } \lbrace n \geq 0 \ | \ \langle G
 \rangle_{n} = \mathcal T \rbrace.
\end{displaymath}
Otherwise, we set $t(G) := \infty$.  We call $t(G)$
the \emph{generation time} of $G$. If $t(G)$ is finite,
we say that $G$ is a \emph{strong generator}. The \emph{Orlov
spectrum} of $\mathcal T$ is the union of all possible generation
times for strong generators of $\mathcal T$.  The \emph{Rouquier
dimension} is the smallest number in the Orlov spectrum.  We say
that a triangulated category, $\mathcal T$ has a \emph{gap} of
length $s$, if $a$ and $a+s+1$ are in the Orlov spectrum but $r$
is not in the Orlov spectrum for $a < r < a+s+1$.
\end{defn}

The first connection to Hodge theory appears in the form of the
following theorem:
\begin{thm}[\cite{BFK}]
Let $X$ be an algebraic variety possessing an isolated
hypersurface singularity. The Orlov spectrum of the category of
singularities of $X$ is bounded by twice the embedding dimension
times the Tjurina number of the singularity.
\label{thm:isohypspecbound}
\end{thm}

After this brief review of theory of spectra and their gaps we connect them with SHS. Let $SHS(X)$ be the Stability Hodge Structure of $D^b(X)$
for given Fano variety $X$, $M(\PP^1, CY)$ is its zero fiber.

\begin{conj}[\cite{KKPS}]
Let $p$  be a point of the divisor $D$ at infinity of the
compactification of $M(\PP^1, CY)$.
The mixed Hodge structures on the completion of the local ring $O_p$,
  where $p$  runs over all components of $D$, determines the spectrum of
$D^{b}(X)$.
\end{conj}

\begin{rem}  The above considerations suggests the existence of a Riemann--Hilbert correspondence for $SHS(X)$ for Fano variety $X$ as well as deep and interesting analytical interpretation of it by analogy with Yang--Mills--Higgs equations.
\end{rem}

As a consequence of the above conjecture we have that SHS satisfy two important properties --- functoriality and strictness. We arrive at:

\begin{conj} The infinite chain  condition (\cite{EKPR})  can be ruled out for the universal coverings of  smooth projective surfaces.

\end{conj}

This is the strongest obstruction to Shafarevich conjecture \cite{EKPR} and SHS gives an approach proving that universal coverings of smooth projective varieties are holomorphically convex.

Observe that the twistor  family of compactified SHS depends on the choice of Landau--Ginzburg model and still   computes some purely  categorical invariants. It is natural to ask whether this family rigidifies the data. In particular we pose:

\begin{question} Does the  twistor  family of compactified SHS of  bounded derived category of coherent sheaves  of a smooth projective variety $X$ recover the fundamental group of $X$?
\end{question}

\section{Multipotential Landau--Ginzburg models and Hodge structures.}

\label{section:multi LG}

In this section we extend the correspondence among  categories  and
Stability Hodge Structures further.  We underscore  the idea  that
 rich geometry of the Landau--Ginzburg models gives a possibility of constructing interesting  Stability Hodge Structures with many filtrations.

\subsection{Multipotential Landau--Ginzburg model for cubic fourfold}

We describe fiberwise compactifications of multipotential Landau--Ginzburg models for cubic fourfold $X$.
This example is representative and illustrates
what we mean by a multipotential Landau--Ginzburg model in general.

The Hori--Vafa toric Landau--Ginzburg for $X$ is
$$
w=\frac{(x+y+1)^3}{xyt_1t_2}+t_1+t_2.
$$
The cubic fourfold is of index 3. So there are two decompositions of its anticanonical divisor:
$3H=H+H+H$ and $3H=2H+H$. Multipotential Landau--Ginzburg models correspond to such decompositions.

\medskip

First we describe compactification for the first decomposition.
We have the family
$$
\frac{(x+y+1)^3}{xyt_1t_2}=w_1,\ \ \ \ \ t_1=w_2,\ \ \ \ \ t_2=w_3,
$$
where $w_i$'s are complex parameters.
After compactifying we
get the family
$$
(x+y+z)^3=w_1w_2w_3xyz
$$
of elliptic curves over $\CC^3$. After blowing up the point
$(0,0,0)$ we get divisor over this point. After that we  resolve the rest of the singularities. The restriction of our family to planes
$w_j=const\neq 0$ is the Landau--Ginzburg model for cubic threefold so
we get the following configuration of singularities.

\begin{enumerate}
    \item Ordinary double points along surface $w_1w_2w_3=27$.
    \item 7 lines forming diagram of type $\widetilde{\mathbb E}_6$ over
    planes $w_1=0$, $w_2=0$, and $w_3=0$.
    \item 5 surfaces over axes $w_1$, $w_2$, $w_3$.
    \item A divisor over (0,0,0).
\end{enumerate}

After projection on the diagonal $\CC^3\to
    \CC$ we get a fiberwise open part of usual Landau--Ginzburg model
    for cubic fourfold. Its fiber over zero consists of  divisor
    described above and an elliptic fibration over the plane passing through the
    origin and orthogonal to the diagonal. The intersection of these
    divisors is an elliptic K3 surface with 3 fibers of type
    $\widetilde{\mathbb E}_6$ corresponding to intersections of this orthogonal plane with
    planes $w_1=0$, $w_2=0$, and $w_3=0$.

\medskip

Now we describe multipotential Landau--Ginzburg model for the second case $2H+H$. We have the family
$$
\frac{(x+y+z)^3}{xyzt_1t_2}=w_1,\ \ \ \ \ \ \ t_1+t_2=w_2.
$$

In other words,

$$
(x+y+z)^3=w_1(w_2-t)txyz
$$
(we denote $t_1$ by $t$ for simplicity).

This family of surfaces can be obtained from decomposition $H+H+H$
by a projection along $w_2+w_3=0$. Indeed, the equation of this
family over $\CC^2$ can be obtained from the equation for the family
over $\CC^3$ by coordinate change $w_2+w_3\to w_2$, $w_3\to t$.

So the singularities are the following.

\begin{enumerate}
    \item Ordinary double points along a curve.
    \item 5 surfaces over axis $w_2=0$.
    \item 17 surfaces over axis $w_1=0$. Their configuration can be
    described as follows: configuration of curves of type
    $\widetilde{\mathbb E}_6$ multiplied by line and two examples of
    configuration of 5 surfaces described above. Each of them are glued by
    intersection of ``pages'' with line of multiplicity 3 on $\widetilde{\mathbb E}_6\times pt$.
    \item A divisor over (0,0,0).
\end{enumerate}

The restriction of this family to the line $w_1=const\neq 0$ is (up
to a multiplication of a potential by a constant) an open part of
Landau--Ginzburg model for cubic threefold. Indeed,
\begin{multline*}
(x+y+z)^3-w_1(w_2-t)txyz=(x+y+z)^3-(\sqrt{w_1}w_2-(\sqrt{w_1}t))(\sqrt{w_1}t)xyz=\\
(x+y+z)^3-(w-t_1)t_1xyz,
\end{multline*}
where $w=\sqrt{w_1}w_2$ and $t_1=\sqrt{w_1}t$.

The restriction to the line $w_2=const\neq 0$ is an open part of
Landau--Ginzburg model for threefold complete intersection of
quadric and cubic. Indeed,
\begin{multline*}
(x+y+z)^3-w_1(w_2-t)txyz=(x+y+z)^3-(w_1w_2^2)\left(1-\frac{t}{w_2}\right)\left(\frac{t}{w_2}\right)xyz=\\
(x+y+z)^3-w(1-t_1)t_1xyz,
\end{multline*}
where $w=w_1w_2^2$ and $t_1=t/w_2$.

On the other hand, compactified (singular) Landau--Ginzburg model
for intersection of quadric and cubic is
$$
(t_1+t_2)^2(x+y+z)^3-wt_1t_2xyz=t_0^2(x+y+z)^3-w(t_0-t_1)t_1xyz,
$$
where $t_0=t_1+t_2$. In the local chart $t_0=1$ we get the family
written down before.

Thus, after compactifying fibers of family corresponding to $2H+H$
we get 4 additional surfaces over  $w_2$ axis  and all together  $21=17+4$ surfaces.

\subsection{Hodge structures with many filtrations}

We now utilize  above construction of multipotential  Landau--Ginzburg models
from the point of view of  twistor families. This part of the paper is highly speculative.

 It is expected that Fukaya--Seidel categories  with many potentials
can be defined similarly to Fukaya--Seidel categories  with one potential.
In this case we have a divisor $S$  of singular fibers and thimbles involved reflect no only geometry of the fibers but geometry of $S$ as well.
In similar way  we can associate to a Fukaya--Seidel category with many potentials a Stability Hodge Structure with a   formal scheme over
 $M(\PP^k, CY)$ as a fiber over zero. The following conjecture (briefly explained in Table~\ref{tab:SHSHS3}) suggests a way of constructing Hodge structures with multiple filtrations.

\begin{conj}[see \cite{KKPS}] The mixed Hodge structure over formal scheme over
 $M(\PP^k, CY)$ as fiber over zero is a mixed Hodge structure with many filtrations.

\end{conj}


\begin{table}[h]
\begin{center}
\begin{tabular}{|c|c|}
\hline
\begin{minipage}[c]{2.5in}
\centering
\medskip

Landau--Ginzburg moduli spaces

\medskip

\end{minipage}
&
\begin{minipage}[c]{2.5in}
\centering
\medskip

Nonabelian Hodge structures

\medskip

\end{minipage}
\\\hline\hline
\begin{minipage}[c]{2.5in}
\centering
\medskip

$M(\PP^1, CY)$

Landau--Ginzburg model with one potential:

The fiber over zero  is a formal scheme over  $M(\PP^1, CY)$,  generic fibers are ${\sf Stab}$.

\medskip

\end{minipage}
&
\begin{minipage}[c]{2.5in}
\centering
\medskip

Twistor family --- Nonabelian Hodge Structure with one weight  filtration.

\medskip

\end{minipage}
\\\hline\hline
\begin{minipage}[c]{2.5in}
\centering
\medskip

Landau--Ginzburg models with $k$ potentials

\medskip

\end{minipage}
&
\begin{minipage}[c]{2.5in}
\centering
\medskip

Generalized  twistor families with $k$ parameters.
\medskip

\end{minipage}
\\\hline\hline
\begin{minipage}[c]{2.5in}
\centering
\medskip

The zero fiber is a formal scheme  over  $M(\PP^k, CY)$, fibers (over a point in $\CC^k$) are ${\sf Stab}$.

\medskip

\end{minipage}
&
\begin{minipage}[c]{2.5in}
\centering
\medskip

Generalized multi twistor family over a  $k$-simplex.

\medskip

\end{minipage}
\\\hline
\begin{minipage}[c]{2.5in}
\centering
\medskip

Extensions

$M(\PP^1, CY)\boxtimes M(\PP^1, CY)$

\medskip

\end{minipage}
&
\begin{minipage}[c]{2.5in}
\centering
\medskip

Extending filtrations $u_i\boxtimes u_j$.

\medskip

\end{minipage}
\\\hline
\end{tabular}
\end{center}
\caption{Creating Hodge structures with multiple filtrations.}
\label{tab:SHSHS3}
\end{table}

\section{Birational transformations and Poisson  varieties}

\label{section:birational}

Discussion from  previous sections suggests that there is a connection between moduli space of Landau--Ginzburg models, generators and birational geometry.

\begin{table}[h]
\begin{center}
\begin{tabular}{|c|c|}
\hline
\begin{minipage}[c]{2.5in}
\centering
\medskip

Landau--Ginzburg model

\medskip

\end{minipage}
&
\begin{minipage}[c]{2.5in}
\centering
\medskip

Stability

\medskip

\end{minipage}
\\\hline\hline
\begin{minipage}[c]{2.5in}
\centering
\medskip

Usual Landau--Ginzburg model

$\includegraphics[width=0.8\nanowidth]{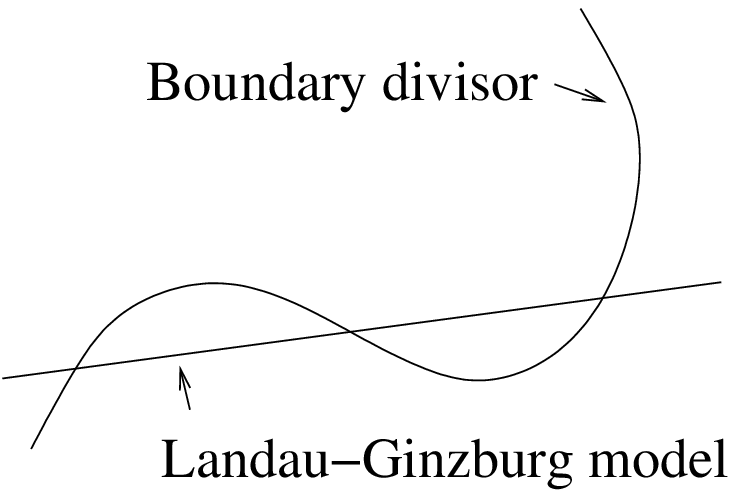}$

\medskip

\end{minipage}
&
\begin{minipage}[c]{2.5in}
\centering
\medskip

$\Omega_X^3$

\medskip

\end{minipage}
\\\hline
\begin{minipage}[c]{2.5in}
\centering
\medskip

$\includegraphics[width=0.8\nanowidth]{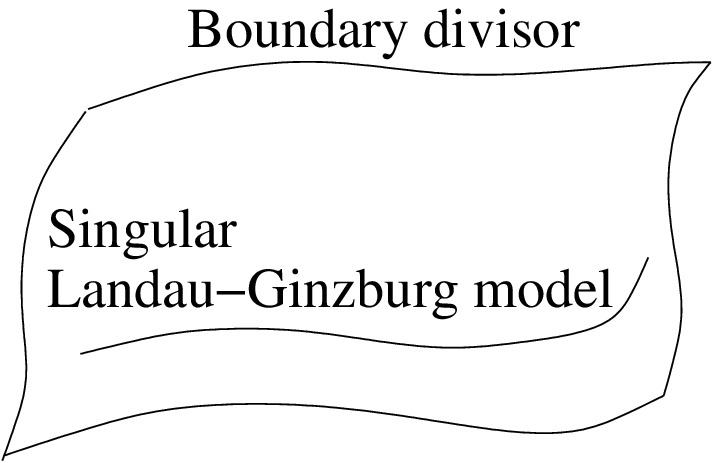}$

\medskip

\end{minipage}
&
\begin{minipage}[c]{2.5in}
\centering
\medskip

$\Omega^3_{X\setminus D}$

$D$ here is a divisor with stratification of singular set.

\medskip

\end{minipage}
\\\hline
\begin{minipage}[c]{2.5in}
\centering
\medskip

$\includegraphics[width=0.8\nanowidth]{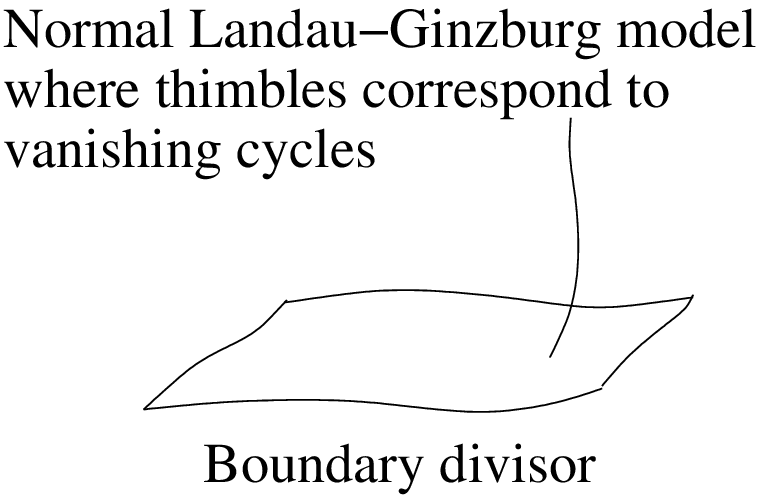}$

\medskip

\end{minipage}
&
\begin{minipage}[c]{2.5in}
\centering
\medskip

$\Omega_D$

${\mathrm Sing}\,D$ stability conditions of the vanishing cycles on $D$.

\medskip

\end{minipage}
\\\hline
\end{tabular}
\end{center}
\caption{Stability Clemens--Schmidt sequence.}
\label{tab:MIAMI1}
\end{table}

Table~\ref{tab:MIAMI1} gives a version of noncommutative Clemens--Schmidt sequence for geometric stability conditions --- log 3-forms.
This table treats the case of three-dimensional Calabi--Yau manifolds (four dimensional
Landau--Ginzburg models) but the situation in general should be rather similar.
In the case at hand (three-dimensional Calabi--Yau manifold) --- the stability conditions are just holomorphic 3-forms. For the quotient category (the category which produces stability conditions of the compactification)  we get stability conditions to be holomorphic 3-forms
vanishing in a stratified way over a divisor $D$.
The vanishing cycles define a subcategory with its own moduli space of stability conditions and the relative (with respect to this subcategory) WCF defining an integrable  system (in general a Poisson variety). The corresponding Landau--Ginzburg models can be seen as follows:

\begin{enumerate}
  \item The Landau--Ginzburg models associated with quotient categories are given by monotonic maps passing through an intersection of many boundary divisors in
$M(\PP^k, CY)$.
  \item
The local categories of vanishing cycles are given by    Landau--Ginzburg models
totally within intersections of divisors.
\end{enumerate}

From the perspective of generators the above splitting corresponds to splitting of generator to union of   generators associated with  subcategory of vanishing cycles and the quotient category. In fact we get a sequence of splittings --- a flag parallel to Okounkov polytopes.

These observations suggest the following conjecture, treated in \cite{DKK}.


\begin{conj} One-parametric families  of Landau--Ginzburg models   parameterize Sarkisov links.
\end{conj}

Recall that Sarkisov links \cite{SARK} are birational maps (birational cobordisms) connecting two Mori fibrations. I our interpretation  Sarkisov links become families connecting circuits.
In fact we have  more general picture on the connections between
moduli spaces of Landau--Ginzburg models and birational geometry. Namely we conjecture that geometry of   moduli spaces of Landau--Ginzburg models  for the mirror of Fano manifold  $X$ determines  its  birational geometry. In particular we see a connection with  relations between Sarkisov links
and then relations between relations and so on. We summarize our picture in Table~\ref{tab:Section 9}. For more details see \cite{DKK}, \cite{BFK3}, \cite{CKP}, \cite{DKLP}.

\begin{table}[h]
\begin{center}
\begin{tabular}{|c|c|}
\hline
\begin{minipage}[c]{2.5in}
\centering
\medskip

Sarkisov programs

\medskip

\end{minipage}
&
\begin{minipage}[c]{2.5in}
\centering
\medskip

Changes in the spaces of stability conditions

\medskip

\end{minipage}
\\\hline\hline
\begin{minipage}[c]{2.5in}
\centering
\medskip

Commutative Sarkisov program: Sarkisov faces.

\medskip

\end{minipage}
&
\begin{minipage}[c]{2.5in}
\centering
\medskip

Wall crossings inside a component of stability conditions.

\medskip

\end{minipage}
\\\hline
\begin{minipage}[c]{2.5in}
\centering
\medskip

Non-commutative Sarkisov program: non-commutative cobordisms.

\medskip

\end{minipage}
&
\begin{minipage}[c]{2.5in}
\centering
\medskip

Passing from one component of stability conditions to another one.

\medskip

\end{minipage}
\\\hline
\end{tabular}
\end{center}
\caption{Birational geometry.}
\label{tab:Section 9}
\end{table}

\medskip

{\bf Acknowledgements.} This paper came out of  a talk the first author gave in G\"okova, Turkey 2011 (Sections~\ref{section:wall crossings},~\ref{section:spectra},~\ref{section:birational}) and discussions thereafter. We are very grateful to the organizers and in particular  S.\,Akbulut, D.\,Auroux and G.\,Mikhalkin for inviting us.

We thank M.\,Kontsevich for sharing his ideas and explaining what SHS should be. Many thanks to D.\,Auroux, G.\,Kerr, C.\,Diemer, D.\,Favero, Y.\,Soibelman, and T.\,Pantev for explaining some of the notions used in the paper.
We thank S.\,Galkin for his explanations of cluster transformations of weak Landau--Ginzburg models.
We thank N.\,Ilten for his explanation of embedded toric degenerations technique.

\end{document}